\documentclass[12pt, a4paper]{article}
\usepackage{amsthm, amsfonts, amsmath, amssymb, float, graphicx, graphics, epstopdf, subfigure, epigraph, appendix, mathrsfs}
\usepackage[left=3 cm,top=3cm,right=3cm, bottom = 3cm]{geometry}
\usepackage[all]{xy}
\usepackage{hyperref}
\usepackage{bm}

\usepackage{color}

\newcommand{\reals}{\mathbb{R}}

\newcommand{\indicator}{\mathbf{1}}
\newcommand{\pr}{\mathbb{P}}

\newcommand{\ex}{\mathbb{E}}

\newcommand{\ep}{\begin{flushright} \vspace{-8.5mm}$\square$ \end{flushright}}
\newcommand{\ga}{[\mathbf{G1}]}
\newcommand{\gb}{[\mathbf{G2}]}
\newcommand{\gc}{[\mathbf{G3}]}

\newcommand{\gspace}{\mathcal{G}}
\newcommand{\gkr}{\gspace_{K,\rad,\cz}}
\newcommand{\rad}{\beta}
\newcommand{\cz}{\gamma}
\newcommand{\dz}{\rho}
\newcommand{\be}{\beta_{\varepsilon}}
\newcommand{\di}{m} 

\newcommand{\ns}{\mspace{0.0mu}}
\newcommand{\vect}[1]{\boldsymbol{#1}}

\providecommand{\abs}[1]{\lvert#1\rvert}
\providecommand{\norm}[1]{\lVert#1\rVert}

\allowdisplaybreaks[1]

\newcounter{lemmacount} 
\newcounter{corcount} 
\newcounter{propcount} 
\newcounter{examples} 
\newcounter{remarks} 

\newtheorem{theorem}{Theorem}
\newtheorem{lemma}[lemmacount]{Lemma}
\newtheorem{proposition}[propcount]{Proposition}
\newtheorem{corollary}[corcount]{Corollary}

\newenvironment{remark}[1][]{\medskip\noindent {\textbf{Remark
\arabic{remarks}.$\,$} \stepcounter{remarks}}}

\begin{document}

\title{On approximation rates for boundary crossing probabilities for the multivariate Brownian motion process}

\author{S.\ McKinlay
\footnote{
Department of Mathematics and Statistics, University of Melbourne, Parkville 3010, Australia. E-mail: s.mckinlay@ms.unimelb.edu.au.}
\ and \ K.\ Borovkov\footnote{
Department of Mathematics and Statistics, University of Melbourne, Parkville 3010, Australia. E-mail: borovkov@unimelb.edu.au.}}

\date{}

\maketitle

\begin{abstract}
\noindent
Motivated by an approximation problem from mathematical finance, we analyse the stability of the boundary crossing probability for the multivariate Brownian motion process, with respect to small changes of the boundary. Under broad assumptions on the nature of the boundary, including the Lipschitz condition (in a Hausdorff-type metric) on its time cross-sections, we obtain an analogue of the Borovkov and Novikov (2005) upper bound for the difference between boundary hitting probabilities for Òclose boundariesÓ in the univariate case. We also obtained upper bounds for the first boundary crossing time densities.
\vspace{0.2cm}

\vspace{0.2cm}

\end{abstract}


\begin{section}{Introduction and main results}

Let $\vect{W} = \{\vect{W}_t = (W_t^{(1)}, \ldots ,  W_t^{(m)})\}_{t \ge 0}$ be the standard $\di$-dimensional Brownian motion process, $\vect{W}_0 = \vect{0}$. For a fixed $T< \infty$, let $\gspace$ be the class of open sets $G \subset (0,T) \times \reals^{\di}$ (the first component representing time), with $(0,\vect{0}) \in \partial G$.

In a number of applied problems (one notable example being  barrier options' pricing), one needs to compute the probability
\begin{equation*}
 P(G) := \pr((t, \vect{W}_t) \in G, t \in (0,T))
\end{equation*}
for $\vect{W}$ to stay within a given set in the time-space $G$ during the time interval $(0,T)$. It is well known that that can be done by solving the respective boundary value problem for the heat equation in $\di$ dimensions (see e.g.\ Section~4.3C in~\cite{Karatzas} for a discussion of the univariate case and \cite{Patie08} for an efficient numerical scheme for computing $P(G)$ for cylindric sets $G$). However, even in the univariate case, a closed form expression for the probability $P(G)$ is only available in a few special cases, the most famous one being when $G$ is specified by a one-sided linear boundary. There is vast literature devoted to different approaches to computing boundary crossing probability and first hitting time densities for the univariate Brownian motion. For a recent bibliography of the published work on that topic see~e.g.~\cite{Kahale}.

Much less was done in the multivariate case. A number of studies have considered the probability
\[
p_{\vect{x}}(C) := \pr\big(\vect{x} + \vect{W}_t \in C, t \in (0,T)\big) \equiv P((0,T) \times (C-\vect{x})), \quad \vect{x} \in C,
\]
when $C \subset\reals^m$ is a cone, the simplest form of which is defined as follows. For $\vect{y} \in \reals^m \backslash \{\vect{0}\}$, let $\theta(\vect{y})$ be the angle between $\vect{y}$ and the point $(1,0, \ldots, 0) \in \reals^m$.  A cone of angle $\alpha \in (0, \pi)$ is defined as $\mathcal{C}_{\alpha} := \{\vect{y} \in \reals^m: 0< \theta(\vect{y}) < \alpha\}$.

It was apparently F.~Spitzer who was the first to consider the probability $p_{\vect{x}}(\mathcal{C}_{\alpha})$ in the two-dimensional case. In~\cite{Spitzer}, he gave an integral transform for the function $p_{\vect{x}}(\mathcal{C}_{\alpha})$. This probability was later computed explicitly in \cite{Iyengar}.

More recently, the case $m \ge 3$ has been considered for ``generalised cones''  defined as follows. If $D$ is a proper open connected subset of the unit sphere $\mathbb{S}^{m-1}$ in $\reals^m$, the  generalised cone  $\mathcal{C}_{D}$ generated by $D$ is the set of all rays emanating from the origin~$\vect{0}$ and passing through~$D$. Under some technical restrictions on~$D$, a representation for $p_{\vect{x}}(\mathcal{C}_{D})$ as an infinite series involving confluent hypergeometric functions and eigenfunctions of the Lapace-Beltrami operator on $\mathbb{S}^{m-1}$ was given in \cite{DeBlassie}. This result was later strengthened in \cite{Banuelos}, where the same analytic formula was shown to hold for a larger class of generalized cones. An alternative technique based on the reflection principle was used in \cite{Muirhead} to compute $p_{\vect{x}}(C)$ in case of ``wedges" $C$.

In the case of general $G \in \gspace$, a possible approach to approximate evaluation of $P(G)$ in nontrivial univariate cases is to approximate $G$ with another set $\widetilde{G} \in \gspace$ for which the computation of $P(\widetilde{G})$ is tractable. For instance, when
\[
G=\{(t,x): t\in (0,T), \ g_- (t) < x < g_+ (t)\},
\]
where $g_-(t) < g_+ (t)$ are smooth enough continuous functions,  one could use a  $\widetilde{G}$ of the same nature but with piece-wise linear boundaries $\widetilde{g}_\pm $ approximating $g_\pm,$ respectively. Recall that, for such boundaries,
the problem of calculating  $P(\widetilde{G})$
reduces (by conditioning on the process' values at the boundaries' ``junction points'') to
calculating the values of $k$-dimensional normal CDFs. For more detail on this technique and a similar approach in the case of the so-called  generalised Daniels' boundaries~\cite{DiNardo01}, see e.g.\ \cite{Borovkov05} and references therein.

To justify the use of such  approximations, however, one must  provide  bounds for the approximation error  $\abs{P(G)-P(\widetilde{G})}$. In the univariate case,  rather tight bounds of such type were obtained for the one-dimensional Brownian motion (see \cite{Borovkov05}) and then extended to time-homogeneous univariate diffusions process (see \cite{Downes08}).

The aim of this note is to extend the outlined approximation approach to the multivariate case and provide bounds for approximation errors. Our results below are also of interest for the theory of boundary value problems for parabolic partial differential equations.

For $H \subset [0,T] \times \reals^{\di}$, let
\begin{equation}\label{Gt}
H_t := \{ \vect{x} = (x_1, \ldots, x_m) \in \reals^{\di}: (t,\vect{x}) \in H\}, \quad t \in [0,T].
\end{equation}
That is, $H_t$ is the time $t$ section of $H$. For $A \subset \reals^{\di}$ and $v \in \reals$, introduce the sets
\begin{equation*}
A^{(v)} := \left\{ \begin{array}{ll}
\{\vect{x} \in \reals^{\di}: \dz(\vect{x},A) < v \}, & v  > 0,\vspace{0.1cm}\\
\{\vect{x} \in \reals^{\di}: \dz(\vect{x},A^c) \le v \}^c, & v \le 0,
\end{array}\right.
\end{equation*}
where $\dz(\vect{x},A) := \inf_{\vect{y} \in A}  \norm{\vect{x}-\vect{y}}$, $\norm{\cdot}$ is the Euclidian norm in $\reals^{\di}$, and $A^c$ is the complement of $A$. For $r>0$, $\vect{x} \in \reals^m$, by $B_r(\vect{x}) := \{ \vect{y} \in \reals^m: \norm{\vect{x}-\vect{y}} < r\}$ we will denote the open ball of radius $r$ with centre at $\vect{x}$.

For $H \subset [0,T] \times \reals^{\di}$, let
\begin{equation*}
H^{(v)} := \{(t,\vect{x}): t \in [0,T], \vect{x} \in H_t^{(v)} \}, \quad v \in \reals.
\end{equation*}

The Hausdorff distance between sets $A, \widetilde{A} \subset \reals^{\di}$ is defined by
\begin{equation*}
\dz_h(A,\widetilde{A}) := \inf\{ \varepsilon>0 :  A \subset \widetilde{A}^{(\varepsilon)} \text{ and } \widetilde{A} \subset A^{(\varepsilon)} \}.
\end{equation*}
It will be convenient for us to use the metric
\begin{equation*}
\dz_H(A,\widetilde{A}) := \max \{ \dz_h(A,\widetilde{A}), \dz_h(A^c,\widetilde{A}^c) \}.
\end{equation*}


For positive numbers $K,\rad,\cz$, introduce the class $\gkr \subset \gspace$ of sets $G$ satisfying the following conditions on their cross-sections.
\begin{itemize}
\item[{$\ga$}] The following Lipschitz condition holds:
\begin{equation*}
\dz_H(G_s, G_t) \le K(t-s), \quad 0 < s<t < T.
\end{equation*}

\item[{$\gb$}] For any $t \in (0,T)$ and $\vect{g} \in \partial G_t$, there exists a ball $B_{\rad}(\vect{y}) \subset  G_t^c$ with $\rad>0$ and $\vect{g} \in \partial B_{\rad}(\vect{y})$.

\item[{$\gc$}] For any $t \in (0,T)$, there exists a $v_0>0$ such that
\begin{equation*}
\ex(1+\norm{\vect{W}_t}; 0 < \dz(\vect{W}_t,G_t^c) < v) < \cz v, \quad 0 < v < v_0.
\end{equation*}
\end{itemize}


The main result of the present paper is the following bound.

\begin{theorem}\label{main}
If $G \in \gkr$, then there exists a $c=c(K,\rad,\cz) < \infty$ such that
\begin{equation}\label{main_bound}
P \left(G^{(\varepsilon)} \right) \le P(G)  + c \varepsilon, \quad \varepsilon > 0.
\end{equation}
\end{theorem}

Note that, in the important special case of convex cross-sections $G_t$, conditions  $\gb$ and  $\gc$ are superfluous, as the following corollary shows.

\begin{corollary}\label{convex_cor}
Assume that $G \in \gspace$ satisfies $\ga$ and $G_t$ is convex for any $t \in (0,T)$. Then $G$ also satisfies $\gb$ with any $\rad>0$ and $\gc$ for some $\cz < \infty$, and so the bound from Theorem \ref{main} holds true.
\end{corollary}

The next result is a trivial consequence of Theorem~\ref{main}. We state it here because  it is the natural multivariate extension of the main bound from~\cite{Borovkov05}.

\begin{corollary}\label{bound_cor}
Suppose $G \in \gkr$. For any $\varepsilon >0$, if sets $G', G'' \in \gspace$ are such that $G \subset G' \subset G^{(\varepsilon)}$, $G \subset G'' \subset G^{(\varepsilon)}$, then
\begin{equation}\label{corr_bound}
\abs{P(G')- P(G'')} < c \varepsilon
\end{equation}
for some constant $c=c(K,\rad,\cz) < \infty$.
\end{corollary}

\begin{remark}
The form of the statement in the above assertion  is somewhat different from the one in the univariate case where we basically estimated the difference $P \big(G^{(\varepsilon)} \big) - P \big(G^{(-\varepsilon)} \big)$. The multivariate situation is noticeably more complicated. In particular, in $\di \ge 2$ dimensions, for a set $G \in \gkr$ it is not necessarily true that  $G^{(-\varepsilon)} \in\gkr$, even if we allow the parameters of the class $\gkr$ in the last instance to be different from those for  the one containing~$G$. One implication of that observation is that,  without  some additional restrictive assumptions, the estimation of $P \big(G^{(\varepsilon)} \big) - P \big(G^{(-\varepsilon)} \big)$ becomes then impossible.
On the other hand, the framework of our Theorem~\ref{main} is quite simple and appears to be the most natural in the multivariate setup.
\end{remark}

\end{section}


\begin{section}{Proofs}

Without loss of generality, we can assume in this section that $T=1$.

For a measurable $H \subset [0,1] \times \reals^{\di}$, let
\begin{equation}\label{tau_set}
\tau(H) := \inf \{t>0: (t,\vect{W}_t) \in \partial H \},
\end{equation}
setting $\tau(H): = 1$ when $\vect{W}_t \notin \partial H_t$, $t \in (0,1)$. Letting $\tau := \tau(G)$, $\tau^{(\varepsilon)} := \tau(G^{(\varepsilon)})$, we have from the Markov property of the Brownian motion that, for $\varepsilon >0$,
\begin{equation}\label{DeG_TPF}
D_{\varepsilon}(G) := P \big(G^{(\varepsilon)} \big) - P(G) = \int_{(0,1)}  \pr \big (\tau^{(\varepsilon)}=1 | \tau = t \big) \pr( \tau \in dt).
\end{equation}


The following proposition, establishing absolute continuity of the distribution of $\tau$ and providing upper bounds for its density, is of independent interest.
\begin{proposition}\label{tau_thm}
The random variable $\tau$ has density $p$ on $(0,1)$ satisfying
\begin{equation*}
p(t) \le 8m^2 \cz \left\{ \begin{array}{ll}
\sqrt{\frac{1}{\pi t}} + \frac{\di-1}{2\rad-Kt} +2K + \frac{2}{t}, & t \in (0,\min \{\rad/K,1 \}),\vspace{0.1cm}\\
\sqrt{\frac{K}{\pi \rad}}  + \frac{ \rad+2}{2t-\rad/K} +  \frac{\di-1}{\rad} + K, &  t \in  [\min \{\rad/K,1 \}, 1).
\end{array}\right.\\
\end{equation*}
\end{proposition}

To prove the proposition, note that,  for any $t\in (0,1),$ setting $\tau_t :=  \inf \{s >t : (s,\vect{W}_s) \in \partial G \},$ one has, for $0<h<1-t$,
\begin{align}
\pr( \tau \in (t, t+h)) &= \int_{G_t} \pr(\tau \in (t,t+h) | \vect{W}_t = \vect{z}) \pr(\vect{W}_t \in d\vect{z})\notag\\
&= \int_{G_t} \pr(\tau > t |\vect{W}_t = \vect{z}) \pr(\tau_t < t+h | \vect{W}_t = \vect{z}) \pr(\vect{W}_t \in d\vect{z}).\label{tau_tpf}
\end{align}
Next we will bound the two factors in the integrand on the right hand side of~\eqref{tau_tpf}. It will be convenient to use the notation
\begin{equation*}
r(\vect{z}) := \dz(\vect{z}, \partial G_t)
\end{equation*}
(for a fixed $t$). The following lemma gives a bound for the first factor.


\begin{lemma}\label{tab_lem}
For $t \in (0,1)$, one has
\begin{equation*}
\frac{\pr(\tau > t  | \vect{W}_t = \vect{z} )}{2 r(\vect{z})} \le \left\{ \begin{array}{ll}
\sqrt{\frac{1}{\pi t}} +  \frac{2(\norm{\vect{z}}+r(\vect{z}))}{t} +\frac{\di-1}{2 \rad-Kt} +2K, & t < \rad/K,\\
\sqrt{\frac{K}{\pi \rad}} + \frac{\norm{\vect{z}}+r(\vect{z})+\rad/2}{t-\rad/(2K)} + \frac{\di-1}{\rad} + K, &  t \ge \rad/K.
\end{array}\right.\\
\end{equation*}
\end{lemma}

The proof of Lemma \ref{tab_lem} uses our next lemma. Before we state the latter, we introduce for $u,v>0$ the (possibly truncated) cones
\begin{equation}\label{Bu}
C(v,u) := \{(s, \vect{x}) \in [0,v] \times \reals^{\di}: \norm{\vect{x}} \le u-Ks \},
\end{equation}
\begin{equation*}
C^*(v,u) := \{(s, \vect{x}) \in [0,v] \times \reals^{\di}: \abs{x_i} \le (u-Ks)/\sqrt{\di}, \: i=1, \ldots, \di\}.
\end{equation*}
Clearly, $C^*(v,u) \subset C(v,u)$.


We will slightly abuse notation by denoting by $\pr_{\vect{x}}$ the distribution on the canonical space corresponding to the Brownian motion process started at the point $\vect{W}_0 = \vect{x} \in \reals^{\di}$ and keeping the notation $\tau(H)$ for the stopping time \eqref{tau_set} for that process.

\begin{lemma}\label{l_tab}
For any $\vect{x},\vect{y} \in \reals^{\di}$ with $\norm{\vect{x}} > \rad$, we have
\small{
\begin{multline*}
\frac{\pr_{\vect{x}}(\tau(C(t,\rad)) > u | \vect{W}_t = \vect{y})}{\norm{\vect{x}}-\rad}
\\
\le \left\{ \begin{array}{ll}
\sqrt{\frac{2}{\pi u}} + 2 \Bigr( \frac{2(\norm{\vect{y}}-\rad)}{t} +\frac{\di-1}{2\rad-Kt} +2K\Bigr)^+, & u \le t/2, \, t < \rad/K,\\
\sqrt{\frac{2}{\pi u}} + 2 \Bigr(\frac{\norm{\vect{y}}-\rad/2}{t-\rad/(2K)} + \frac{\di-1}{\rad} + K \Bigr)^+, &  u \le \rad/(2K), \, t \ge \rad/K,
\end{array}\right.
\end{multline*}
}
where $x^+ := \max \{0,x\}$.
\end{lemma}

Note that the above upper bounds agree at $t = \rad/K$.

To prove Lemma \ref{l_tab}, we will require the following two additional lemmas. For a univariate process $X = \{X_t\}_{t \ge 0}$ and $x \in \reals$, set
\begin{equation}\label{tau_a}
\eta_x(X) := \inf\{t \ge 0: X_t = x \}.
\end{equation}

\begin{lemma}\label{drift_lem}
Let $\{W_t\}_{t \ge 0}$ be the standard univariate Brownian motion given on a filtered probability space,  $\{Y_t\}_{t \ge 0}$ a continuous adapted process on the same space. Let $X^{(1)}_t$ and $X^{(2)}_t$ be strong unique solutions of the stochastic differential equations {\em (SDEs)}
\begin{equation*}
dX^{(i)}_t= a_i (t,X^{(i)}_t,Y_t)dt + dW_t, \quad X^{(i)}_0 = x_0, \quad i=1,2,
\end{equation*}
where $a_i$ are continuous. Suppose that, for a given $l < x_0$, one has $a_1(t,x,y) < a_2(t,x,y)$ for all $(t,x) \in [0,\infty) \times (l,\infty)$, $y \in \reals$. Then $X^{(1)}_t < X^{(2)}_t$ a.s.\ for all $t \in (0,\eta_l(X^{(1)}))$.
\end{lemma}

The proof of Lemma \ref{drift_lem} below follows the argument proving a somewhat weaker assertion of Lemma $4$ on p.120 of \cite{gs}.

\begin{proof}
Define the continuously differentiable function
\begin{equation*}
\Delta(t) := X^{(2)}_t-X^{(1)}_t = \int_0^t (a_2 (s,X^{(2)}_s,Y_s)-a_1 (s,X^{(1)}_s,Y_s)) \, ds, \quad t<\eta_l.
\end{equation*}
Then, for all points $t<\eta_l$ with $\Delta(t) = 0$, we have that $X^{(1)}_t = X^{(2)}_t$, and so at these points
\begin{equation*}
\Delta'(t) = a_2 (t,X^{(2)}_t,Y_t)-a_1 (t,X^{(1)}_t,Y_t) >0.
\end{equation*}
In particular, we have $\Delta(0) = 0$, $\Delta'(0+) >0$. Therefore we can find a $\delta>0$ such that $\Delta(t)>0$ for all $0<t \le \delta$. Now suppose the set $\{t \in (0,\eta_l): \Delta(t)=0\}$ is not empty. Then for $t_1 := \inf \{t \in (0,\eta_l): \Delta(t)=0\}$ we have $\Delta(t_1) = 0$, $\Delta'(t_1) >0$, and so there exists a $\delta_1>0$ such that $\Delta(t)<0$ for $t \in [t_1-\delta_1,t_1]$. Therefore $\Delta(t)$ changes signs on the interval $[\delta, t_1- \delta_1]$, i.e., it takes on the value zero there, which contradicts the definition of $t_1$. We conclude that $\{t \in (0,\eta_l): \Delta(t)=0\}$ is empty a.s., and since $\Delta(t) >0$ for sufficiently small $t$, $\Delta(t) >0$ for all $t \in(0,\eta_l)$ as required.
\end{proof}


Recall that $\{W_t\}_{t \ge 0}$ is the standard univariate Brownian motion process.

\begin{lemma}\label{linear_bound}
For $c \in \reals$ and $\varepsilon>0$,
\begin{equation*}
\pr (\sup_{0 \le s \le t} (W_s - cs) < \varepsilon) \le \varepsilon \biggr( \sqrt{\frac{2}{\pi t}} + 2 c^+\biggr).
\end{equation*}
\end{lemma}

\begin{proof}
The probability on the left hand side above is known explicitly (see e.g.\ $1.1.4$ on p.250 of~\cite{Borodin}): denoting by $\Phi$ the standard normal distribution function,
\begin{align*}
\pr (\sup_{0 \le s \le t} (W_s - cs) < \varepsilon) &= \Phi \bigr(c \sqrt{t}+\varepsilon/\sqrt{t} \bigr) - e^{-2c \varepsilon} \Phi \bigr(c \sqrt{t}-\varepsilon/\sqrt{t} \bigr)\\
&\le \Phi \bigr(c \sqrt{t}+\varepsilon/\sqrt{t} \bigr) - e^{-2 c^+ \varepsilon} \Phi \bigr(c \sqrt{t}-\varepsilon/\sqrt{t} \bigr)\\
&\le \Phi \bigr(c \sqrt{t}+\varepsilon/\sqrt{t} \bigr) - \Phi \bigr(c \sqrt{t}-\varepsilon/\sqrt{t} \bigr) + (1-e^{-2 c^+ \varepsilon})\\
&\le \sup_{x \in \reals} \Phi'(x) \times 2\varepsilon/\sqrt{t} + 2 c^+ \varepsilon\\
&= \varepsilon \biggr(\sqrt{\frac{2}{\pi t}} + 2 c^+ \biggr).
\end{align*}
\end{proof}


{\em Proof of Lemma~\ref{l_tab}.} Let $\vect{B} = \{\vect{B}_s = (B^{(1)}_s, \ldots, B^{(\di)}_s) \}_{0 \le s \le t}$ be an $\di$-dimensional Brownian bridge process starting at $\vect{x} \in \reals^{\di}$ at time $0$ and ending at $\vect{y} \in \reals^{\di}$ at time $t$.

In order to use Lemma \ref{drift_lem}, we will now derive an SDE for the radial process $S_s := \norm{\vect{B}_s}$ of $\vect{B}$. Recall that $\vect{B}$ satisfies the SDE
\begin{equation}\label{bb_sde}
d\vect{B}_s =\frac{\vect{y}- \vect{B}_s}{t-s} ds + d \vect{W}_s, \quad 0 < s < t,
\end{equation}
(see e.g.\ p.64 in \cite{Borodin}) By It\^o's formula, the squared radial process has stochastic differential
\begin{equation}\label{rad_sde}
dS^2_s = 2 \sum_{i=1}^{\di} B^{(i)}_s dB^{(i)}_s + \di \, ds, \quad 0 < s < t.
\end{equation}
Setting $\vect{\xi_s} := \vect{B}_s/S_s$, we have $\norm{\vect{\xi_s}} \equiv 1$ and therefore
\begin{equation}\label{xi_bound}
\xi_s(\vect{y}) := \vect{\xi_s}\vect{y}^T \le \norm{\vect{y}},
\end{equation}
where $\vect{y}^T$ denotes the transpose of $\vect{y}$. Then, for $0 < s < t$, we have from \eqref{bb_sde} and \eqref{rad_sde} that
\begin{align*}
dS_s^2 &= 2 \sum_{i=1}^{\di} B^{(i)}_s \biggr( \frac{y_i- B^{(i)}_s}{t-s} ds + dW^{(i)}_s \biggr) +\di \, ds\\
&= 2 \left(\frac{\vect{B}_s\vect{y}^T-S_s^2}{t-s} + \frac{\di}{2} \right)ds + 2\vect{B}_s d \vect{W}_s^T\\
&= 2 \left(\frac{S_s \xi_s(\vect{y}) -S_s^2}{t-s} + \frac{\di}{2} \right)ds + 2S_s \vect{\xi}_s d \vect{W}_s^T\\
&= 2 \left(\frac{S_s \xi_s(\vect{y})-S_s^2}{t-s} + \frac{\di}{2} \right)ds + 2S_s d \widetilde{W}_s,
\end{align*}
where $\{\widetilde{W}_t\}_{t \ge 0}$ is a standard univariate Brownian motion, and the last equality follows from Theorem $8.4.2$ in \cite{Oksendal}.

Using the above SDE for $\{S^2_s\}$ and It\^o's formula with $f(x)=\sqrt{x}$, we have
\begin{align*}
dS_s &= f'(S^2_s)dS^2_s + \frac{1}{2}f''(S^2_s)(dS^2_s)^2\\
&= \frac{1}{2 S_s} \biggr[ 2 \left(\frac{S_s \xi_s(\vect{y})-S_s^2}{t-s} + \frac{\di}{2} \right)ds + 2S_s d \widetilde{W}_s \biggr] - \frac{1}{8 S_s^3} (2 S_s)^2 ds\\
&= \left(\frac{\xi_s(\vect{y})-S_s}{t-s} + \frac{\di-1}{2 S_s} \right)ds + d \widetilde{W}_s, \quad 0<s<t.
\end{align*}

Now introduce, for a fixed $a<\norm{\vect{x}}$ and $t_0 \in (0,t),$ the reference process
\begin{equation*}
\overline{S}_s := \norm{\vect{x}} + \left( \frac{\norm{\vect{y}}-a}{t-t_0}+\frac{\di-1}{2a} \right)s + \widetilde{W}_s, \quad s \ge 0,
\end{equation*}
Since $\norm{\vect{y}} \ge \xi_s(\vect{y})$ by \eqref{xi_bound}, Lemma \ref{drift_lem} implies that, for all $s~\in~[0,\min \{ t_0, \eta_a(S) \}]$, one has $\overline{S}_s \ge S_s$ a.s.

Consider first the case $t \ge \rad/K$  and set $t_0 := \rad/(2K)$, $a := \rad/2$. Then, for all $u \le  t_0$, one has
\begin{align}
\pr_{\vect{x}}(\tau(&C(t,\rad)) > u | \vect{W}_t = \vect{y})\notag\\
&= \pr \Bigr(\inf_{0 \le s \le u} (S_s -\rad+Ks)> 0\Bigr)\notag\\
&\le \pr \Bigr(\inf_{0 \le s \le u} (\overline{S}_s - \rad + Ks) >0\Bigr)\notag\\
&= \pr \biggr[\inf_{0 \le s \le u} \biggr(\norm{\vect{x}}+ \biggr( \frac{\norm{\vect{y}}-\rad/2}{t-\rad/(2K)}+\frac{\di-1}{\rad} \biggr)s +W_s - \rad + Ks \biggr)>0 \biggr]\notag\\
&= \pr \biggr[\sup_{0 \le s \le u} \biggr( W_s - \biggr(  \frac{\norm{\vect{y}}-\rad/2}{t-\rad/(2K)} + \frac{\di-1}{\rad} + K \biggr)s \biggr) <  \norm{\vect{x}}-\rad \biggr]\notag\\
&\le (\norm{\vect{x}}-\rad) \Biggr[\sqrt{\frac{2}{\pi u}} + 2 \biggr(\frac{\norm{\vect{y}}-\rad/2}{t-\rad/(2K)} + \frac{\di-1}{\rad} + K \biggr)^+\Biggr]\notag
\end{align}
by Lemma \ref{linear_bound}.

Now consider the case $t < \rad/K$ and set $t_0 := t/2$, $a := \rad-Kt/2$. Then, for all $u \le t_0$, we have
\begin{align}
\pr_{\vect{x}}(\tau(&C(t,\rad)) > u | \vect{W}_t = \vect{y} )\notag\\
&\le \pr \Bigr(\inf_{0 \le s \le u} (\overline{S}_s - \rad + Ks) >0 \Bigr)\notag\\
&= \pr \biggr [\inf_{0 \le s \le u} \biggr(\norm{\vect{x}}+ \biggr( \frac{\norm{\vect{y}}-\rad+Kt/2}{t/2}+\frac{\di-1}{2\rad-Kt} \biggr)s +W_s - \rad + Ks \biggr)>0 \biggr ]\notag\\
&= \pr \biggr [\sup_{0 \le s \le u} \biggr( W_s - \biggr(  \frac{2(\norm{\vect{y}}-\rad)}{t} +\frac{\di-1}{2\rad-Kt} + 2K\biggr)s \biggr) <  \norm{\vect{x}}-\rad \biggr]\notag\\
&\le (\norm{\vect{x}}-\rad) \Biggr [\sqrt{\frac{2}{\pi u}} + 2 \biggr( \frac{2(\norm{\vect{y}}-\rad)}{t} +\frac{\di-1}{2 \rad-Kt} + 2K\biggr)^+ \Biggr ].\notag
\end{align}
Lemma~\ref{l_tab} is proved. \ep


{\em Proof of Lemma~\ref{tab_lem}.} Fix $t \in (0,1)$ and $\vect{z} \in G_t$. Reversing the time for the conditional Brownian motion process, we have for $t' \in (0,t)$,
\begin{align}
\pr(\tau > t  | \vect{W}_t = \vect{z} ) &= \pr(\vect{W}_s \in G_s, s \in (0,t) | \vect{W}_t = \vect{z} )\notag\\
&= \pr_{\vect{z}} (\vect{W}_s \in G_{t-s}, s \in (0,t) | \vect{W}_t = \vect{0})\notag\\
&\le \pr_{\vect{z}} (\vect{W}_s \in G_{t-s}, s \in (0,t') | \vect{W}_t = \vect{0}).\label{tab1}
\end{align}
One can clearly choose a $\vect{b} \in \partial G_t$ such that $\dz(\vect{z}, \vect{b}) = r(\vect{z})$. Then by condition $\gb$ there exists a ball $B_{\rad}(\vect{c})  \subset G_t^c$ with $\vect{b} \in \partial B_{\rad}(\vect{c})$. Using $B_{\rad}(\vect{c})$ as the base for the cone
\begin{equation*}
C := (0,\vect{c}) + C(t,\rad),
\end{equation*}
it follows from Lipschitz condition $\ga$ that $G_{t-s} \subset C_s^c$, $s \in [0,t]$. Therefore
\begin{align}
\pr_{\vect{z}} (\vect{W}_s \in G_{t-s}, s \in (0,t') | \vect{W}_t = \vect{0}) &\le \pr_{\vect{z}} (\vect{W}_s \in C_s^c, s \in (0,t') | \vect{W}_t = \vect{0})\notag\\
&= \pr_{\vect{z}-\vect{c}} (\tau(C(t,\rad)) > t' | \vect{W}_t = -\vect{c}).\label{conebound}
\end{align}
Since $\norm{-\vect{c}} \le \norm{\vect{z}}+r(\vect{z})+\rad$, we immediately obtain the bounds stated in Lemma~\ref{tab_lem} from Lemma~\ref{l_tab} with
\begin{equation*}
t' = u := \biggr\{ \begin{array}{ll}
t/2, & t < \rad/K,\\
\rad/(2K), &  t \ge \rad/K.
\end{array}\\
\end{equation*}
\ep


Now we will turn to bounding the second factor on the right hand side of~\eqref{tau_tpf}.

\begin{lemma}\label{inball_lem}
For $t \in (0,1)$, $\vect{z} \in G_t$, one has
\begin{equation*}
\pr(\tau_t < t+h | \vect{W}_t = \vect{z}) \le 2m  \exp \biggr(\frac{r(\vect{z}) K}{\di} - \frac{r(\vect{z})^2}{2h\di} \biggr), \quad h \in (0,\min\{r(\vect{z})/K,1-t \}).
\end{equation*}
\end{lemma}

\begin{proof}
For $h$ from the specified interval, setting $\overline{\Phi}(x) := 1-\Phi(x)$, we have
\begin{align}
\pr(\tau_t < t+h | \vect{W}_t = \vect{z}) & \le \pr(\tau(C(h,r(\vect{z})))< h)\notag\\
& \le \pr(\tau(C^*(h,r(\vect{z})))< h)\notag\\
& \le 2\di \, \pr (\sup_{0 \le s \le h} (W(s)+Ks/\sqrt{\di}) \ge r(\vect{z})/\sqrt{\di})\notag\\
&= 2\di \int_0^h \frac{r(\vect{z})/\sqrt{\di}}{\sqrt{2\pi}s^{3/2}}  \exp \biggr(\frac{-(r(\vect{z})/\sqrt{\di}-Ks/\sqrt{\di})^2}{2s} \biggr)  ds\label{kendall}\\
&= r(\vect{z}) \sqrt{\frac{2\di}{\pi}} e^{r(\vect{z}) K/\di} \int_0^h s^{-3/2} \exp \biggr(-\frac{r(\vect{z})^2}{2s\di} - \frac{K^2 s}{2\di} \biggr)  ds\notag\\
&\le r(\vect{z}) \sqrt{\frac{2\di}{\pi}} e^{r(\vect{z}) K/\di} \int_0^h s^{-3/2}  \exp \biggr(-\frac{r(\vect{z})^2}{2s\di} \biggr)  ds\notag\\
&= 4\di \, e^{r(\vect{z}) K/\di} \overline{\Phi}(r(\vect{z})/\sqrt{h\di})\label{subu}\\
&\le 2\di  \exp \biggr(\frac{r(\vect{z}) K}{\di} - \frac{r(\vect{z})^2}{2h\di} \biggr),
\label{qbound}
\end{align}
where \eqref{kendall} follows from Kendall's formula (see e.g.\ relation 2.0.2 on p.295 of \cite{Borodin}), \eqref{subu} follows by making the substitution $u = r(\vect{z})/\sqrt{s\di}$, and \eqref{qbound} follows by using the bound $\overline{\Phi}(x) \le \frac{1}{2} e^{-x^2/2}$, $x>0$. The lemma is proved.
\end{proof}


{\em Proof of Proposition~\ref{tau_thm}.}
Suppose that $t < \rad/K$. Then from \eqref{tau_tpf} and the bounds derived in Lemmas \ref{tab_lem}, \ref{inball_lem}, we have
\begin{align}
\pr( \tau \in (t, t+h)) &= \int_{G_t} \pr(\tau > t |\vect{W}_t = \vect{z}) \pr(\tau_t < t+h | \vect{W}_t = \vect{z}) \pr(\vect{W}_t \in d\vect{z})\notag\\
& \le 4\di \int_{G_t}  \biggr(\sqrt{\frac{1}{\pi t}} + \frac{2(\norm{\vect{z}}+r(\vect{z}))}{t} +\frac{\di-1}{2 \rad-Kt} +2K \biggr)  \notag\\
&\quad \times  r(\vect{z}) \exp \biggr(\frac{r(\vect{z}) K}{\di} - \frac{r(\vect{z})^2}{2h\di} \biggr) \pr(\vect{W}_t \in d\vect{z})\notag\\
& = 4\di \biggr( \biggr(\sqrt{\frac{1}{\pi t}} +\frac{\di-1}{2\rad-Kt} +2K \biggr)I_1 + \frac{2}{t}( I_2 + I_3) \biggr)\label{bnd1},
\end{align}
where
\begin{align*}
I_1 &:= \int_{G_t} r(\vect{z}) \exp \biggr( \frac{r(\vect{z})K}{\di} - \frac{r(\vect{z})^2}{2h\di} \biggr) \pr(\vect{W}_t \in d\vect{z}),\\
I_2 &:= \int_{G_t} r(\vect{z}) \norm{\vect{z}} \exp \biggr( \frac{r(\vect{z})K}{\di} - \frac{r(\vect{z})^2}{2h\di} \biggr) \pr(\vect{W}_t \in d\vect{z}),\\
I_3 &:= \int_{G_t} r(\vect{z})^2  \exp \biggr( \frac{r(\vect{z})K}{\di} - \frac{r(\vect{z})^2}{2h\di} \biggr) \pr(\vect{W}_t \in d\vect{z}).
\end{align*}

Set
\begin{equation*}
Z := r(\vect{W}_t) \indicator_{\{\vect{W}_t \in G_t\}},
\end{equation*}
where $\indicator_E$ is the indicator of event $E$. Then, for $u(x) := xe^{xK/\di-x^2/(2h\di)}$, we have
\begin{equation*}
I_1 = \ex Z \exp \biggr\{ \frac{ZK}{\di}-\frac{Z^2}{2h\di} \biggr\} = \int_0^{\infty} u(x) d \nu(x),
\end{equation*}
where $\nu(x)$ is the distribution function of $Z$. Integrating by parts and using the bound $\nu(x) < \cz x$, $x \in (0,v_0),$  from $\gc$, we obtain
\begin{align*}
I_1 &= [\nu(x)u(x)]_0^{\infty} - \int_0^{\infty} \nu(x) du(x) = -\int_0^{\infty} \nu(x) du(x)\notag\\
&=  \int_0^{\infty} \nu(x) \biggr( \frac{x^2}{h\di} - \frac{xK}{\di}-1 \biggr) \exp \biggr( \frac{xK}{\di} - \frac{x^2}{2h\di} \biggr) dx\\
&<  \frac{1}{h\di} \int_0^{\infty} \nu(x) x^2 \exp \biggr( \frac{xK}{\di} - \frac{x^2}{2h\di} \biggr) dx = \frac{1}{h\di} \left( \int_0^{v_0} + \int_{v_0}^{\infty} \right)(\cdots ) \, dx\\
&<  \frac{1}{h\di} \left( \cz \int_0^{\infty} x^3 \exp \biggr( \frac{xK}{\di} - \frac{x^2}{2h\di} \biggr) dx +  \int_{v_0}^{\infty} x^2 \exp \biggr( \frac{xK}{\di} - \frac{x^2}{2h\di} \biggr) dx \right)\\
&= \cz h \di \int_0^{\infty} s^3 e^{sK\sqrt{h/\di} - s^2/2} \, ds +\sqrt{h \di} \int_{v_0/\sqrt{h \di}}^{\infty} s^2 e^{sK\sqrt{h/\di} - s^2/2} \, ds\notag\\
&= 2\cz h\di + o(h),
\end{align*}
where the second last relation follows by making the substitution $s = x/\sqrt{h\di}$.

Using $\gc$ and following the same steps as above, we conclude that
\begin{equation*}
I_2 < 2\cz h\di + o(h).
\end{equation*}
Finally, it is even simpler to show that $I_3 = o(h)$.

Then, from \eqref{bnd1}, we have
\begin{equation*}
\pr( \tau \in (t, t+h)) < 8\di^2 \cz h  \biggr(\sqrt{\frac{1}{\pi t}} + \frac{\di-1}{2\rad-Kt} +2K + \frac{2}{t}  \biggr) + o(h).
\end{equation*}
It follows that $\tau$ has an absolutely continuous distribution specified by a density $p$ satisfying
\begin{equation*}
p(t) \le 8\di^2 \cz  \biggr(\sqrt{\frac{1}{\pi t}} + \frac{\di-1}{2\rad-Kt} +2K + \frac{2}{t}  \biggr), \quad t \in (0,\min \{ \rad/K,1 \}).
\end{equation*}

Now consider the case when $t \ge \rad/K$. Then, from \eqref{tau_tpf} and the bounds derived in Lemmas  \ref{tab_lem}, \ref{inball_lem}, we have
\begin{align*}
\pr( \tau & \in (t, t+h)) = \int_{G_t} \pr(\tau > t |\vect{W}_t = \vect{z}) \pr(\tau_t < t+h | \vect{W}_t = \vect{z}) \pr(\vect{W}_t \in d\vect{z})\\
& \le 4\di \int_{G_t} \biggr(\sqrt{\frac{K}{\pi \rad}} + \frac{\norm{\vect{z}}+r(\vect{z})+\rad/2}{t-\rad/(2K)} + \frac{\di-1}{\rad} + K \biggr)\\
&\quad \times r(\vect{z}) \exp \biggr(\frac{r(\vect{z}) K}{\di} - \frac{r(\vect{z})^2}{2h\di} \biggr)\pr(\vect{W}_t \in d\vect{z})\\
& = 4\di \biggr[ \biggr(\sqrt{\frac{K}{\pi \rad}}  + \frac{ \rad}{2t-\rad/K} + \frac{\di-1}{\rad}  +  K \biggr)I_1 + \frac{I_2 + I_3}{t-\rad/(2K)} \biggr],
\end{align*}
and therefore, for $t \in [\rad/K,1)$,
\begin{equation*}
\pr( \tau \in (t, t+h)) \le 8\di^2 h \cz \biggr(\sqrt{\frac{K}{\pi \rad}}  + \frac{\rad+2}{2t-\rad/K} +  \frac{\di-1}{\rad} + K  \biggr) .
\end{equation*}
As above, it follows that $\tau$ has density $p$ satisfying
\begin{equation*}
p(t) \le 8\di^2 \cz \biggr(\sqrt{\frac{K}{\pi \rad}}  + \frac{\rad+2}{2t-\rad/K} +  \frac{\di-1}{\rad} + K  \biggr), \quad t \in [\min \{ \rad/K,1 \},1).
\end{equation*}
Proposition~\ref{tau_thm} is proved. \ep


To prove Theorem \ref{main}, we will also use the following lemma that provides a bound for the integrand on the right hand side of \eqref{DeG_TPF}.

\begin{lemma}\label{bessel_lem}
For $\vect{x} \in \reals^{\di}$  and $0< r<\norm{\vect{x}}$, we have
\begin{equation*}
\frac{\pr_{\vect{x}}(\tau(C(t,r)) > t)}{2(\norm{\vect{x}}-r)} \le \left\{ \begin{array}{ll}
\sqrt{\frac{1}{\pi t}} + \frac{\di-1}{2 r-Kt} + K, & t < r/K,\\
\sqrt{\frac{K}{\pi r}} + \frac{\di-1}{r} + K, & t \ge r/K.
\end{array}\right.\\
\end{equation*}
\end{lemma}

\begin{proof}
Denote by $R = \{R_s\}_{s \ge 0}$ an $\di$-dimensional Bessel process started at $\norm{\vect{x}}$ at time~$0$. One can stipulate that
\begin{equation*}
R_s = \sqrt{ \bigr(\norm{\vect{x}}+W^{(1)}_s \bigr)^2 + \bigr(W^{(2)}_s \bigr)^2 + \cdots + \bigr(W^{(\di)}_s \bigr)^2}, \quad s \ge 0,
\end{equation*}
and so
\begin{equation}\label{bessel2bm}
\pr_{\vect{x}} ( \tau(C(t,r)) > t ) = \pr (\inf_{0 \le s \le t} (R_s - r+Ks)> 0).
\end{equation}
As is well-known (see e.g.\ p.148 in~\cite{Oksendal}), $R$ satisfies the SDE
\begin{equation*}
dR_s = \frac{\di-1}{2R_s} ds + d\widetilde{W}_s, \quad s>0, \quad R_0 = \norm{\vect{x}},
\end{equation*}
$\{\widetilde{W}_s\}_{s \ge 0}$ being a standard univariate Brownian motion process.

Consider first the case $t < r/K$ and let
\begin{equation*}
\overline{R}_s := \norm{\vect{x}}+\frac{\di-1}{2r-Kt}s + \widetilde{W}_s, \quad s \in [0,t/2].
\end{equation*}
Then, by Lemma \ref{drift_lem}, we have $\overline{R}_s \ge R_s$ a.s.\ for all $s \in [0,\min \{t/2, \eta_{r-Kt/2}(R) \}]$ (cf.\ \eqref{tau_a}). From here, \eqref{bessel2bm} and Lemma \ref{linear_bound},   one has
\begin{align}
\pr_{\vect{x}} (\tau(C(t,r)) > t ) &\le \pr \Bigr(\inf_{0 \le s \le t/2} (\overline{R}_s - r + Ks)>0 \Bigr)\notag\\
&= \pr \Bigr(\inf_{0 \le s \le t/2} \Bigr(\norm{\vect{x}} + \frac{\di-1}{2r-Kt}s + \widetilde{W}_s - r + Ks \Bigr ) > 0 \Bigr)\notag\\
&= \pr \biggr[ \sup_{0 \le s \le t/2} \Bigr(\widetilde{W}_s - \Bigr(\frac{\di-1}{2 r-Kt} + K \Bigr)s \Bigr) < \norm{\vect{x}} - r \biggr]\notag\\
&\le 2(\norm{\vect{x}}-r) \biggr(\sqrt{\frac{1}{\pi t}} + \frac{\di-1}{2 r-Kt} + K \biggr).\notag
\end{align}

Now consider the case $t \ge r/K$ and let
\begin{equation*}
\overline{R}_s := \norm{\vect{x}}+\frac{\di-1}{r}s + \widetilde{W}_s, \quad s \in [0,r/(2K)].
\end{equation*}
Then, by Lemma \ref{drift_lem}, we have $\overline{R}_s \ge R_s$ a.s.\ for all $s \in [0,\min \{r/(2K), \eta_{r/2}(R) \}]$, and so from \eqref{bessel2bm} and Lemma \ref{linear_bound}, a similar derivation yields the bound
\begin{equation*}
\pr_{\vect{x}} ( \tau(C(t,r)) > t ) \le 2(\norm{\vect{x}}-r) \biggr(\sqrt{\frac{K}{\pi r}} + \frac{\di-1}{r} + K \biggr),
\end{equation*}
as required.
\end{proof}


Now we can complete the proof of Theorem~\ref{main}. The integrand on the right hand side of \eqref{DeG_TPF} has the form
\begin{align}
\pr \big(\tau^{(\varepsilon)}=1 | \tau = t \big) &= \int_{\partial G_t} \pr \big(\tau^{(\varepsilon)}=1, \vect{W}_t \in d \vect{z} | \tau = t \big)\notag\\
&= \int_{\partial G_t} \pr_{\vect{z}} \big( \vect{W}_s \in G_{t+s}^{(\varepsilon)}, s \in (0,1-t) \big)  \pr (\vect{W}_t \in d \vect{z} | \tau = t )\label{tpf_dg}.
\end{align}
For any $\vect{z} \in \partial G_t$, by $\gb$ there is a point $\vect{y}$ such that $B_{\rad}(\vect{y}) \subset G_t^c$ and $\vect{z} \in \partial B_{\rad}(\vect{y})$. Clearly, $B_{\be}(\vect{y}) \subset \big(G_t^{(\varepsilon)} \big)^c$, where $\be := \rad-\varepsilon$ (we assume without loss of generality that $\varepsilon < \beta/2$), and so $G_t^{(\varepsilon)} \subset \big(B_{\be}(\vect{y})  \big)^c$. By condition $\ga$, we then also have
\begin{equation*}
G_{t+s}^{(\varepsilon)} \subset G_t^{(\varepsilon+Ks)} \subset  \big(B_{\be-Ks}(\vect{y})  \big)^c = \big((0,\vect{y}) + C(1-t,\be) \big)^c_s, \quad s \le \frac{\be}{K},
\end{equation*}
using notation \eqref{Bu}. Therefore, since $\norm{\vect{z}-\vect{y}} - \be = \varepsilon$, one has from Lemma~\ref{bessel_lem} that
\begin{align*}
\pr_{\vect{z}} \big( \vect{W}_s \in G_{t+s}^{(\varepsilon)}, s \in (0,1-t) \big) & \le \pr_{\vect{z}-\vect{y}} ( \tau(C(1-t,\be))>1-t )\\
& \le 2 \varepsilon \left\{ \begin{array}{ll}
\sqrt{\frac{1}{\pi (1-t)}} + \frac{\di-1}{2\be-K(1-t)} + K, & 1-t < \be/K,\\
\sqrt{\frac{K}{\pi \be}} + \frac{\di-1}{\be} + K, & 1-t \ge \be/K.
\end{array}\right.
\end{align*}
Now it follows from \eqref{DeG_TPF} and \eqref{tpf_dg} that
\begin{align*}
D_{\varepsilon}(G) &\le  2 \varepsilon \int_0^{1-\be/K} \left( \sqrt{\frac{K}{\pi \be}} + \frac{\di-1}{\be} + K \right) p(t) \, dt\\
& \quad + 2 \varepsilon \int_{1-\be/K}^1 \left(\sqrt{\frac{1}{\pi (1-t)}} + \frac{\di-1}{2\be-K(1-t)} + K \right) p(t) \, dt\\
&\le 2 \varepsilon  \left( \sqrt{\frac{K}{\pi \be}} +  \frac{\di-1}{\be} + K \right) + 2 \varepsilon \int_0^1 \frac{p(t)\, dt}{\sqrt{\pi(1-t)}}.
\end{align*}
Here
\begin{align*}
\int_0^1 \frac{p(t)\, dt}{\sqrt{\pi(1-t)}} &\le \sqrt{\frac{2}{\pi}} + \int_{1/2}^1 \frac{p(t)\, dt}{\sqrt{\pi(1-t)}}\\
&\le \sqrt{\frac{2}{\pi}} + 8  \di^2 \cz \sqrt{\frac{2}{\pi}} \left(\sqrt{\frac{K}{\pi \rad}}  + 2(\rad+2) +  \frac{\di-1}{\rad} + K \right) =: c^*,
\end{align*}
where the second inequality follows by assuming without loss of generality that $\rad/K < 1/2$ and applying the second bound from Proposition \ref{tau_thm}. Recalling that $\be \ge \rad/2$, we conclude that
\begin{equation*}
D_{\varepsilon}(G) \le 2\varepsilon \left( c^* + \sqrt{\frac{2K}{\pi \rad}} +  \frac{2(\di-1)}{\rad} + K \right).
\end{equation*}
Theorem~\ref{main} is proved.
\ep

{\em Proof of Corollary~\ref{convex_cor}.} Condition {$\gb$} is clearly satisfied (with arbitrary large $\beta >0$) due to the convexity of $G_t$, so we only need to verify {$\gc$}.

We can assume without loss of generality that there exists a $\delta >0$ such that  $B_{\delta}(\vect{0}) \subset(\partial G)_0$  (for otherwise it is easy to see that, in view of {$\ga$}, one has $P(G)=0$ and the whole problem becomes trivial).

Consider first the case $t> t_0: = \delta/ (2K) $ (assuming that $t_0 <1$).

Introduce the sequence of spherical layers $C_k := B_k(\vect{0}) \backslash B_{k-1}(\vect{0})$, $k=1,2,\ldots $ As the cross-section $G_t$ is convex, it follows from Cauchy's surface area formula (see e.g.\ Theorem~5.5.2 on p.56 in \cite{Klain}) that the ($(m-1)$-dimensional) surface area of $\partial(G_t \cap B_k(\vect{0}))$ does not exceed the surface area of $B_k(\vect{0})$ which is equal to $k^{m-1} \omega_{m-1}$, where $\omega_{m-1}$ is the unit sphere area.

Therefore, again using the convexity of $G_t$, it follows that, for any $\varepsilon >0,$  the volume of $V_k := \bigr(G_t \backslash G_t^{(-\varepsilon)} \bigr) \cap C_k$ does not exceed $\varepsilon k^{m-1} \omega_{m-1}$. As the maximum value of the density of $\vect{W}_t$ on $V_k$ less than or equal to the density's value on $\partial B_{k-1}(\vect{0}),$ we conclude that
\begin{align*}
\ex(1+\norm{\vect{W}_t}; \vect{W}_t \in G_t \backslash G_t^{(-\varepsilon)} ) &= \sum_{k=1}^{\infty} \ex(1+\norm{\vect{W}_t}; \vect{W}_t \in V_k)\\
&\le \frac{\varepsilon  \omega_{m-1}}{(2 \pi t)^{m/2}} \sum_{k=1}^{\infty} (1+k)k^{m-1} e^{-(k-1)^2/2t}\\
&\le \cz \varepsilon
\end{align*}
for some   $\cz = \cz(m,t_0) <\infty ,$ as we only consider $t \in ( t_0,1)$.

Now turn to the case $t \in (0,t_0)$. By  $\ga$, for such $t$ one has $B_{\delta/2}(\vect{0}) \subset G_{t }$. Choosing $v_0 := \delta/4$ (which we can assume to be less that one without loss of generality), one can employ the same argument as above but  with the spherical layers $C_k := B_{k+ \delta/4}(\vect{0}) \backslash  B_{k-1+ \delta/4}(\vect{0})$, $k=1,2,\ldots,$ making use of the observation that the maximum value of the density of  $\vect{W}_t$ on the ``innermost layer" $C_1$ is bounded for $t \in (0,t_0)$.  The corollary is proved.\\
\ep

\end{section}


\vspace{0.3cm}
\noindent \textbf{Acknowledgements.} This research was supported by the ARC Centre of Excellence for Mathematics and Statistics of Complex Systems, the Maurice Belz Trust and ARC Discovery Grant DP150102758.
The first author wishes to thank the School of Mathematical Sciences at Queen Mary, University of London, for providing a visiting position while this research was undertaken.



\begin{thebibliography}{99}

\bibitem{Banuelos}
{\sc Ba\~nuelos, R. and Smits, R. G.} (1997).
Brownian motion in cones.
{\it Probab. Th. Rel. Fields}
{\bf 108}, 299--319.

\bibitem{Borodin}
{\sc Borodin, A. N. and Salminen, P.} (2002).
{\it Handbook of Brownian motion--Facts and formulae,}
2nd edn. Birkhauser, Basel.

\bibitem{Borovkov05}
{\sc Borovkov, K. and Novikov, A.} (2005).
Explicit bounds for approximation rates of boundary crossing probabilities for the wiener process.
{\it J. Appl. Prob.}
{\bf 42}, 82--92.


\bibitem{DeBlassie}
{\sc DeBlassie, D.} (1987).
Exit times from cones in $\reals^n$ of Brownian motion.
{\it Probab. Th. Rel. Fields}
{\bf 74}, 1--29.

\bibitem{DiNardo01}
{\sc Di Nardo, E., Nobile, A. G., Pirozzi, E. and Ricciardi, L. M.} (2001).
A computational approach to first-passage-time problems for
Gauss-Markov processes.
{\it Adv.\  Appl.\ Probab.}
{\bf 33}, 453--482.

\bibitem{Downes08}
{\sc Downes, A. and Borovkov, K.} (2008).
First passage densities and boundary crossing probabilities for diffusion processes.
{\it Meth. Comp. Appl. Prob.}
{\bf 10}, 621--644.

\bibitem{gs}
{\sc Gihman, I. I. and Skorohod, A. V.} (1972).
{\it Stochastic differential equations,}
Springer-Verlag.

\bibitem{Iyengar}
{\sc Iyengar, S.} (1985)
Hitting lines with two-dimensional Brownian motion.
{\it SIAM J. Math. Anal.}
{\bf 45}, 983--989.

\bibitem{Kahale}
{\sc Kahale, N.}  (2008)
Analytic crossing probabilities for certain barriers by Brownian motion.
{\it Ann.\ Appl.\ Probab.}
{\bf 18}, 1424--1440.

\bibitem{Karatzas}
{\sc Karatzas, I. and Shreve S. E.} (1998)
{\it  Brownian motion and stochastic calculus.}
Springer-Verlag.

\bibitem{Klain}
{\sc Klain, D. A. and Rota, G.-C.} (1997)
{\it Introduction to geometric probability.}
Cambridge University Press.

\bibitem{Muirhead}
{\sc Muirhead, S.} (2011)
{\it Financial Geometry: Pricing multi-asset barrier options using the generalised reflection principle.}
Honours Thesis, The University of Melbourne.

\bibitem{Oksendal}
{\sc \O ksendal, B.} (2007).
{\it Stochastic differential equations,}
6th Ed. Springer.

\bibitem{Patie08}
{\sc Patie, P. and Winter, C.} (2008)
First exit time probability for multidimensional diffusions: A PDE-based approach.
{\it J. Comput. Appl. Math.}
{\bf 222}, 42--53.

\bibitem{Peres2010}
{\sc Morters, P. and Peres, Y.} (2010)
{\it  Brownian motion.}
Cambridge University Press.

\bibitem{Spitzer}
{\sc Spitzer, F.} (1958)
Some theorems concerning 2-dimensional brownian motion.
{\it Trans. Amer. Math. Soc.}
{\bf 87}, 187--197.

\end{thebibliography}
\end{document}